\documentclass[12pt,a4paper,reqno]{amsart}
\usepackage[mathcal]{eucal}
\usepackage{amssymb}
\usepackage[nice]{nicefrac}
\usepackage{hyperref}
\usepackage[normalem]{ulem}
\usepackage{xcolor}

\theoremstyle{plain}
\newtheorem{theorem}{Theorem}[section]
\newtheorem{lemma}[theorem]{Lemma}

\theoremstyle{definition}

\numberwithin{equation}{section}

\makeatletter
\newcommand{\AlignFootnote}[1]{%
    \ifmeasuring@
    \else
        \footnote{#1}%
    \fi
}
\makeatother

\title{On conciseness of the word in Olshanskii's example}

\author[M. Pintonello]{Matteo Pintonello} 
\address{Matteo Pintonello: Department of Mathematics, Euskal Herriko Unibertsitatea UPV/EHU, 48940 Leioa, Spain}
\email{matteo.pintonello@ehu.eus}

\author[P. Shumyatsky]{Pavel Shumyatsky} 
\address{Pavel Shumyatsky: Department of Mathematics, University of Brasilia, Brasilia DF, Brazil}
\email{pavel@unb.br}

\date{\today}

\thanks{This paper was written while the second author was visiting Department of Mathematics of the University of the Basque Country. He expresses his sincere gratitude to the department for excellent hospitality.
The first author was supported by the Spanish Government project PID2020-117281GB-I00, partially by FEDER funds, by the Basque Government project IT483-22, and grant FPI-2018 of the Spanish Government. He is extremely thankful to G. Fern\'andez-Alcober and M. Casals for the useful suggestions.
}


\begin{document}

\maketitle

\begin{abstract}
A group-word $w$ is called concise if the verbal subgroup $w(G)$ is finite whenever $w$ takes only finitely many values in a group $G$. It is known that there are words that are not concise. In particular, Olshanskii gave an example of such a word, which we denote by $w_o$. The problem whether every word is concise in the class of residually finite groups remains wide open. In this note we observe that $w_o$ is concise in residually finite groups. Moreover, we show that $w_o$ is strongly concise in profinite groups, that is, $w_o(G)$ is finite whenever $G$ is a profinite group in which $w_o$ takes less than $2^{\aleph_0}$ values.
\end{abstract}

\section{Introduction}

A group-word $w$ is called concise in the class of groups $\mathcal C$ if the verbal subgroup $w(G)$ is finite whenever $w$ takes only finitely many values in a group $G\in\mathcal C$. In the sixties Hall raised the problem whether every word is concise in the class of all groups but in 1989 S. Ivanov \cite{Iva89} solved the problem in the negative. More precisely, he showed that if $n\geq 10^{10}$ is an odd integer and $p\geq 5000$ is a prime, there is a torsion-free group in which the word $[[x^{pn},y^{pn}]^n,y^{pn}]^n$ takes only two values.

Another noteworthy word was introduced by Olshanskii in \cite{Ols85}. For positive integers $d$ and $n$, set

$$v(x,y)=[[x^d,y^d]^d,[y^d,x^{-d}]^d]$$
and
$$w_o(x,y)=[x,y]v(x,y)^n[x,y]^{\varepsilon_1}v(x,y)^{n+1}\cdots [x,y]^{\varepsilon_{h-1}}v(x,y)^{n+h-1},$$
where 
$$\varepsilon_{10k+1}=\varepsilon_{10k+2}=\varepsilon_{10k+3}=\varepsilon_{10k+5}=\varepsilon_{10k+6}=1$$
$$\varepsilon_{10k+4}=\varepsilon_{10k+7}=\varepsilon_{10k+8}=\varepsilon_{10k+9}=\varepsilon_{10k+10}=-1$$
for $k=0,1,\ldots,(h-1)/10$ and $h\equiv 1 \mod 10$, $h>50000$.

Olshanskii showed that the parameters $n$ and $d$ can be chosen in such a way that the word $w_o$ has several remarkable properties.   In particular, the variety of groups where $w_o$ is a law contains infinite non-abelian groups while all finite groups in the variety are abelian. Moreover, 
the word is not concise in the class of all groups (see \cite[p.\ 439]{Ols91B}).

From now on, throughout the paper $w_o$ stands for Olshanskii's word with the above properties.

The problem whether all words are concise in residually finite groups remains wide open (cf. Segal \cite[p.\ 15]{Seg09} or Jaikin-Zapirain \cite{Jai08}). In recent years several new positive results with respect to this problem were obtained (see \cite{AcSh14, GuSh15, FAPi23, FASh18, DeMoShEng1, DeMoShEng2, DeMoSh19}). In the present paper we observe that the word $w_o$ is concise in residually finite groups (Theorem \ref{main1}).

A natural variation of the notion of conciseness for profinite groups was introduced in \cite{DeKlSh20}: the word $w$ is strongly concise in a class of profinite groups $\mathcal{C}$ if the verbal subgroup $w(G)$ is finite in any group $G \in \mathcal{C}$ in which $w$ takes less than $2^{\aleph_0}$ values. Here and throughout the paper, whenever $G$ is a profinite group we write $w(G)$ to denote the {\it closed} subgroup generated by $w$-values. A number of new results on strong conciseness of group-words can be found in \cite{DeKlSh20,Det23,AzSh21,KhSh23,HPS23}. In this note we will show that the word $w_o$ is strongly concise in profinite groups (Theorem \ref{main2}).

\section{Preliminaries}

By a subgroup of a profinite group we always mean a closed subgroup and by a homomorphism of profinite groups a continuous homomorphism.

\begin{lemma}\cite[Lemma 4]{DeMoShEng1}\label{fromm} Let $w$ be a word and $G$ a group such that the set of  $w$-values in $G$ is finite with at most $m$ elements. Then the order of $w(G)'$ is $m$-bounded.
\end{lemma}

\begin{lemma}[\cite{DeKlSh20} Lemma 2.2]\label{lem: finite conjugacy class}
    Let $G$ be a profinite group and $g\in G$ be an element whose conjugacy class $g^G$ contains less than $2^{\aleph_0}$ elements. Then $g^G$ is finite.
\end{lemma}

The following lemma is straightforward.
\begin{lemma}\label{lem: in metabelian groups} Let $v(x,y)$ be as in the introduction.
In every group where $v$ is a law, and in particular in metabelian groups, the values of the word $w_o$ coincide exactly with the values of the commutator word $[x,y]$.
\end{lemma}

Using the above lemma and the fact that every finite non-abelian group contains a non-abelian metabelian subgroup, the following result is immediate.

\begin{theorem}[Lemma 29.1 of \cite{Ols91B}]
    Every finite group where $w_o$ is a law is abelian.
\end{theorem}
 Olshanskii proved that the variety of groups satisfying the law $w_o\equiv1$ contains non-abelian infinite groups. Moreover, he established the following theorem.

\begin{theorem}[Theorem 39.7 of \cite{Ols91B}] \label{th: olshanskii abelian}
    There is a group $G$ where $w_o$ takes a single nontrivial value, but $w_o(G)$ is infinite.
\end{theorem}

Thus, it follows that the word $w_o$ is not concise.
\section{Main results}

A group-word $w$ is called {\it boundedly} concise in the class of groups $\mathcal C$ if the verbal subgroup $w(G)$ is finite of $(m,w)$-bounded order whenever $w$ takes at most $m$ values in a group $G\in\mathcal C$. It is known, due to Fern\'andez-Alcober and Morigi \cite{FAMo10}, that every word that is concise in the class of all groups is in fact boundedly concise. On the other hand, it is an open problem whether the same phenomenon holds for words that are concise in the class of residually finite groups.
We can now prove the following theorem.
\begin{theorem}\label{main1}
    The word $w_o$ is boundedly concise in residually finite groups.
\end{theorem}
\begin{proof} Let $m$ be a positive integer and 
     $G$ a residually finite group in which $w_o$ takes $m$ values. In view of Lemma \ref{fromm} there is a number $f_1$ depending only on $m$ such that $|w_o(G)'|\leq f_1$.
    If $Q$ is a finite homomorphic image of $G$, observe that the quotient $Q/w_o(Q)$ is abelian by Theorem \ref{th: olshanskii abelian}. Hence $Q/w_o(Q)'$ is metabelian. Lemma \ref{lem: in metabelian groups} now implies that  $Q/w_o(Q)'$ has at most $m$ commutators. Note that the commutator word is boundedly concise (see for example \cite{SegSh99} for an explicit bound), so $|w_o(Q/w_o(Q'))|\leq f_2$, where $f_2$ is a number depending only on $m$. Hence $|w_o(Q)|\leq f_1f_2$. This holds for every finite  homomorphic image $Q$ of $G$ so we deduce that $|w_o(G)|\leq f_1f_2$.
\end{proof}
We will now prove that the word $w_o$ is strongly concise in profinite groups.

Start with the following lemma.

\begin{lemma}\label{simple}
 Let $G$ be a profinite group topologically isomorphic to a Cartesian product of finite simple groups.
If the word $w_o$ takes less than $2^{\aleph_0}$ values in $G$, then $G$ is finite.
\end{lemma}
\begin{proof} Write $G=\prod_{i\in I} S_i$, where the factors $S_i$ are finite simple groups. Since every finite non-abelian group contains a non-abelian metabelian group, it follows that $w(S_i)$ is nontrivial for any $i\in I$. We need to show that the index set $I$ is finite. Assume by contradiction that $I$ is infinite.  Choose a nontrivial $w_o$-value $c_i\in S_i$ for each $i\in I$. Observe that for each subset $J\subseteq I$ the product $c_J=\prod_{i\in J} c_i$ is a $w_o$-value. If $J_1\neq J_2$, then $c_{J_1}\neq c_{J_2}$ and therefore $G$ contains at least $2^{\aleph_0}$ distinct $w_o$-values, a contradiction. \end{proof}

\begin{lemma}\label{prosoluble}
 Let $G$ be a prosoluble group. If the word $w_o$ takes less than $2^{\aleph_0}$ values in $G$, then the commutator subgroup $G'$ is finite.
\end{lemma}
\begin{proof}  
 By Lemma \ref{lem: in metabelian groups}, $w_o(G/G'')=G'/G''$. Moreover, as the commutator word is strongly concise in profinite groups \cite{DeKlSh20}, $G'/G''$ is finite. Therefore there exists a finite set $T$ of $w_o$-values such that $G'=\langle T \rangle G''$. Note that by Lemma \ref{lem: finite conjugacy class} each element of $T$ has finitely many conjugates in $G$. So we can choose $T$ in such a way that the subgroup $\langle T\rangle$ is normal in $G$. Set $\bar G=G/\langle T\rangle$. Observe that $\bar G$ is a prosoluble group with the property that $\bar G'=\bar G''$. It follows that $\bar G$ is abelian and so $G'=\langle T\rangle$.
 
Now by Lemma \ref{lem: finite conjugacy class}, for each $t\in T$ we have that $[\langle T \rangle :C_{G}(t)]\leq \infty$, whence $[\langle T \rangle:Z(\langle T \rangle)]\leq \infty$. By Schur's Theorem, the commutator subgroup $\langle T \rangle'$ is finite. This implies that $G$ is finite-by-metabelian. Factoring out $G''$ we can assume that $G$ is metabelian and apply again Lemma \ref{lem: in metabelian groups}. Since the commutator word is strongly concise, we conclude that $G'$ is finite, as required.
\end{proof}

It is well known that if $K$ is a finite group, then there exists a series
$$
1=K_0\leq K_1\leq \cdots \leq K_{2h+1}=K
$$
of normal subgroups of $K$ such that $K_{i+1}/K_{i}$ is soluble (possibly trivial) if $i$ odd and a direct product of non-abelian simple groups if $i$ is even. The number of insoluble factors in this series is called the insoluble length $\lambda(K)$ of $K$.
Theorem 1.4 of \cite{KhSh15} implies that if the Sylow 2-subgroup of $K$ is soluble with derived length $l$, then $\lambda(G)$ is bounded in terms of $l$ only. We are now ready to complete the proof that the word $w_o$ is strongly concise in profinite groups.

\begin{theorem}\label{main2}
 Let $G$ be a profinite group in which the word $w_o$ takes less than $2^{\aleph_0}$ values. Then the verbal subgroup $w_o(G)$ is finite.
\end{theorem}

\begin{proof}

Choose a $2$-Sylow subgroup $P$ of $G$. In view of Lemma \ref{prosoluble} observe that $P$ is soluble, say of derived length $l$.
It follows that if $Q$ is any finite homomorphic image of $G$, the insoluble length $\lambda(Q)$ is bounded in terms of $l$ only. 
Applying now a combination of Lemma 2 and Lemma 3 of \cite{Wil83}, we obtain that the group $G$ has a normal series of finite length
\begin{equation}\label{eq: series p-solvable}
1=G_0\leq G_1\leq \cdots \leq G_{h}=G
\end{equation}
each of whose factors is either prosoluble or a Cartesian product of non-abelian finite simple groups.

Lemma \ref{simple} shows that the non-prosoluble factors of the above series are finite. Let $C$ be the intersection of the centralizers in $G$ of the non-prosoluble factors, that is $$C=\{g\in G;\ [G_{i+1},g]\leq G_i\text{ whenever } G_{i+1}/G_i\text{ is not prosoluble}\}.$$
Since the non-prosoluble factors are finite, it follows that $C$ is open in $G$. Note that all the sections of the series obtained by intersecting $C$ with the series \ref{eq: series p-solvable} are prosoluble. Therefore we conclude that $C$ is prosoluble and hence Lemma \ref{prosoluble} tells us that $C'$ is finite. Passing to the quotient $G/C'$ without loss of generality we can assume that our group $G$ is virtually abelian. Now the result follows from a theorem of Detomi that says that every word is strongly concise in the class of virtually nilpotent profinite groups \cite{Det23}.

Alternatively, we can argue in a more direct way. Since $G$ is virtually abelian, the verbal subgroup $w_o(G)$ is generated by its normal open abelian subgroup $A$ and finitely many $w_o$-values, say $g_1,\dots,g_s$. Here, by Lemma \ref{lem: finite conjugacy class}, the centralizer $C_G(g_i)$ is open for every $i=1,\dots,s$. It follows that $B=A\cap C_G(g_1)\cap\dots\cap C_G(g_s)$ is open in $w_o(G)$. Obviously, $B\leq Z(w_o(G))$ so we conclude that $Z(w_o(G))$ is open in $w_o(G)$. The Schur theorem tells us that the commutator subgroup $w_o(G)'$ is finite so we can pass to the quotient $G/w_o(G)'$ and without loss of generality assume that $w_o(G)$ is abelian. But then $G$ is metabelian and the result is immediate from Lemma \ref{prosoluble}.

\end{proof}

\end{document}